\documentclass[a4paper]{amsproc}
\usepackage{amsmath}
\numberwithin{equation}{section}
\usepackage{amsthm}
\usepackage{amsfonts}
\usepackage{amssymb}
\usepackage{mathrsfs}
\usepackage{tikz}
\usepackage{environ}
\usepackage{elocalloc}
\usepackage{url}
\usepackage{caption}
\usepackage{CJKutf8}
\usepackage[normalem]{ulem}
\usepackage{xcolor}
\usepackage{etoolbox}
\usepackage{caption}
\usepackage{mathtools}
\usepackage{dcpic}
\usepackage{tikz-cd}
\usepackage{hyperref}
\usepackage{enumitem}
\usepackage{stmaryrd}
\usetikzlibrary{positioning}
\usetikzlibrary{arrows,chains,positioning,scopes,quotes}
\usetikzlibrary{decorations.markings}

\title{The Existence of Embedded $G$-Invariant Minimal Hypersurface}
\author{Zhenhua Liu}
\date{}
\dedicatory{Dedicated to Xunjing Wei}
\begin{document}
	\begin{abstract}
		For a compact connected Lie group $G$ acting as isometries {of cohomogeneity not equal to $0$ or $2$} on a compact orientable Riemannian manifold $M^{n+1},$ we prove the existence of a nontrivial embedded $G$-invariant minimal hypersurface, that is smooth outside a set of Hausdorff dimension at most $n-7.$
	\end{abstract}
	\maketitle
	
	\newcommand\st{\bgroup\markoverwith{\textcolor{red}{\rule[0.5ex]{2pt}{0.4pt}}}\ULon}
	
	\newcommand{\ai}{\alpha}
	\newcommand{\be}{\beta}
	\newcommand{\sm}{\setminus}
	\newcommand{\Ga}{\Gamma}
	\newcommand{\ga}{\gamma}	
	\newcommand{\de}{\delta}
	\newcommand{\De}{\Delta}
	\newcommand{\e}{\epsilon}
	\newcommand{\lam}{\lambda}
	\newcommand{\Lam}{\Lamda}
	\newcommand{\om}{\omega}
	\newcommand{\Om}{\Omega}
	\newcommand{\si}{\sigma}
	\newcommand{\Si}{\Sigma}
	\newcommand{\vp}{\varphi}
	\newcommand{\cb}{\mathcal{B}}
	\newcommand{\rh}{\rho}
	\newcommand{\ta}{\theta}
	\newcommand{\Ta}{\Theta}
	\newcommand{\W}{\mathcal{O}}
	\newcommand{\ps}{\psi}
	
	\newcommand{\mf}[1]{\mathfrak{#1}}
	\newcommand{\ms}[1]{\mathscr{#1}}
	\newcommand{\mb}[1]{\mathbb{#1}}
	\newcommand{\cd}{\cdots}
	
	\newcommand{\s}{\subset}
	\newcommand{\es}{\varnothing}
	\newcommand{\cp}{^\complement}
	\newcommand{\bu}{\bigcup}
	\newcommand{\ba}{\bigcap}
	\newcommand{\ito}{\uparrow}
	\newcommand{\dto}{\downarrow}		
	\newcommand{\di}{\operatorname{div}}

	\newcommand{\ti}[1]{\tilde{#1}}
	\newcommand{\la}{\langle}
	\newcommand{\ra}{\rangle}
	\newcommand{\ov}[1]{\overline{#1}}
	\newcommand{\no}[1]{\left\lVert#1\right\rVert}
	\DeclarePairedDelimiter{\cl}{\lceil}{\rceil}
	\DeclarePairedDelimiter{\fl}{\lfloor}{\rfloor}
	\DeclarePairedDelimiter{\ri}{\la}{\ra}

	\newcommand{\du}{^\ast}
	\newcommand{\ma}{\mathbf{M}}
	\newcommand{\pf}{_{\#}}
	\newcommand{\vol}{\textnormal{Vol}}
	\newcommand{\is}{\cong}
	\newcommand{\co}{\mathcal{CO}^G}
	\newcommand{\n}{\lhd}
	\newcommand{\m}{^{-1}}
	\newcommand{\ts}{\otimes}
	\newcommand{\ip}{\lrcorner}
	\newcommand{\op}{\oplus}
	\newcommand{\xr}{\xrightarrow}
	\newcommand{\xla}{\xleftarrow}
	\newcommand{\xhl}{\xhookleftarrow}
	\newcommand{\xhr}{\xhookrightarrow}
	\newcommand{\mi}{\mathfrak{m}}
	\newcommand{\wi}{\widehat}
	\newcommand{\sch}{\mathcal{S}}

	\newcommand{\w}{\wedge}
	\newcommand{\X}{\mathfrak{X}}
	\newcommand{\pd}{\partial}
	\newcommand{\dx}{\dot{x}}
	\newcommand{\dr}{\dot{r}}
	\newcommand{\dy}{\dot{y}}
	\newcommand{\dth}{\dot{theta}}
	\newcommand{\pa}[2]{\frac{\pd #1}{\pd #2}}
	\newcommand{\na}{\nabla}
	\newcommand{\dt}[1]{\frac{d#1}{d t}\bigg|_{ t=0}}
	\newcommand{\ld}{\mathcal{L}}

	\newcommand{\N}{\mathbb{N}}
	\newcommand{\R}{\mathbb{R}}
	\newcommand{\Z}{\mathbb{Z}}
	\newcommand{\Q}{\mathbb{Q}}
	\newcommand{\C}{\mathbb{C}}
	\newcommand{\bh}{\mathbb{H}}
	
	\newcommand{\lix}{\lim_{x\to\infty}}
	\newcommand{\li}{\lim_{n\to\infty}}
	\newcommand{\infti}{\sum_{i=1}^{\infty}}
	\newcommand{\inftj}{\sum_{j=1}^{\infty}}
	\newcommand{\inftn}{\sum_{n=1}^{\infty}}	
	\newcommand{\snz}{\sum_{n=-\infty}^{\infty}}	
	\newcommand{\ie}{\int_E}
	\newcommand{\ir}{\int_R}
	\newcommand{\ii}{\int_0^1}
	\newcommand{\sni}{\sum_{n=0}^\infty}
	\newcommand{\ig}{\int_{\ga}}
	\newcommand{\pj}{\mb{P}}
	
	\newcommand{\io}{\textnormal{ i.o.}}
	\newcommand{\aut}{\textnormal{Aut}}
	\newcommand{\out}{\textnormal{Out}}
	\newcommand{\diam}{\textnormal{diam}}
	\newcommand{\inn}{\textnormal{Inn}}
	\newcommand{\mult}{\textnormal{mult}}
	\newcommand{\ord}{\textnormal{ord}}
	\newcommand{\F}{\mathcal{F}}
	\newcommand{\V}{\mathbf{V}}	
	\newcommand{\II}{\mathbf{I}}
	\newcommand{\ric}{\textnormal{Ric}}
	\newcommand{\sef}{\textnormal{II}}
	\newcommand{\hm}{\mathcal{H}}
	\newcommand{\mv}{\mathcal{V}}
	\newcommand{\vd}{\mf{d}}
	\newcommand{\sv}{\mathcal{V}_\infty}
	\newcommand{\wh}{\Rightarrow}
	\newcommand{\eq}{\Leftrightarrow}
	
	\newcommand{\eqz}{\setcounter{equation}{0}}
	\newcommand{\se}{\subsection*}
	\newcommand{\ho}{\textnormal{Hom}}
	\newcommand{\ds}{\displaystyle}
	\newcommand{\tr}{\textnormal{tr}}
	\newcommand{\id}{\textnormal{id}}
	\newcommand{\lp}{\llcorner}
	\newcommand{\im}{\textnormal{im}}
	\newcommand{\ev}{\textnormal{ev}}
	
	\newcommand{\gl}{\mf{gl}}
	\newcommand{\sll}{\mf{sl}}
	\newcommand{\su}{\mf{su}}
	\newcommand{\so}{\mf{so}}
	\newcommand{\ad}{\textnormal{ad}}

	\theoremstyle{plain}
	\newtheorem{thm}{Theorem}[section]
	\newtheorem{lem}[thm]{Lemma}
	\newtheorem{prop}[thm]{Proposition}
	\newtheorem{cor}{Corollary}
	\newtheorem*{claim}{Claim}
	\newtheorem*{pro}{Proposition}
	
	\theoremstyle{definition}
	\newtheorem{defn}{Definition}[section]
	\newtheorem{conj}{Conjecture}[section]
	\newtheorem{exmp}{Example}[section]
	
	\theoremstyle{remark}
	\newtheorem*{rem}{Remark}
	\newtheorem*{note}{Note}
	
	\tableofcontents
	
	\section{Introduction}
	Based on the continuous version of min-max construction in \cite{DT}, we prove the following theorem,
	\begin{thm}\label{t0.1}
		Let $G$ be a compact connected Lie group acting as isometries {of cohomogeneity not equal to $0$ or $2$} on a compact connected Riemannian manifold $M^{n+1}$ of dimension $(n+1)$ without boundary. If $n\ge 2$ and the action of $G$ is not transitive, then there exists an embedded minimal hypersurface $\Si^n\s M^{n+1}$ that is invariant under the action of $G$. Moreover, $\Si^n$ has no boundary and is smooth outside a set $\textnormal{Sing}\Si$ of Hausdorff dimension at most $n-7.$
	\end{thm}
	Invariant means for any $s\in \Si^n,$ and all $g\in G,$ we have $g.s\in \Si^n$. In other words, $\Si^n$ is a union of orbits. The statement of our Theorem\autoref{t0.1} differs from Theorem 0.1 in \cite{DT} only in that our minimal hypersurface $\Si$ is invariant under $G$-actions. Our assumptions on the actions are very mild. 
	
	Jon Pitts and J. H. Rubinstein have announced a version of finite group action invariant min-max constructions in \cite{P1} and \cite{P2}, though a full proof has never appeared to the author's knowledge. Daniel Ketover has developed in \cite{DK} an equivariant min-max for finite group actions on three-dimensional manifolds, and thereby proved many results and conjectures appeared in \cite{P1} and \cite{P2}. This work is inspired by his approach, especially in the part of the existence theory of invariant stationary and invariant almost minimizing varifolds. However, the regularity theory regarding $G$-invariant replacements are done using vastly different methods. We relied on an old argument of Lawson in \cite{BL} that minimizing among invariant surfaces implies minimizing among all surfaces in our case.
	\subsection{Remarks}
	One might wonder whether we have better control of regularity. The answer is yes in some cases. Note that by invariance of $\Si$, we can push forward smooth tangent planes. Thus, if a point $s$ is in the singular set $\text{Sing}\Si,$ then $g.s$ must also belong to $\text{Sing}\Si.$ This implies the singularity sets must also be the union of orbits, so $\text{Sing}\Si$ consists of orbits of dimension no larger than $n-7.$ Consequently, in some practical cases like those in \cite{HL}, the singular set could roughly have dimensions of order $\frac{n}{2}$ or $\frac{n}{3},$ which are the largest orbit dimensions below $n-7$.
	
	However, generically, we can say nothing more about the regularity. By \cite{HL}, the projection of any minimal hypersurface $\Si^n\s M^{n+1}$ to the orbit space $M/G$ with a weighted metric $V^{2/k}\ov{g}$ will be a minimal hypersurface in the open manifold $M^{\text{principal}}/G$. Thus, the principal orbit part of $\Si^n$ can be reduced to a generic minimal hypersurface on $M/G$, which imposes the $n-7$ regularity. Meanwhile, that metric $V^{2/k}\ov{g}$ vanishes on singular orbits, so it provides no information about the singular orbits parts of $\Si^n.$ Since the union of singular orbits can have dimension up to $n-1,$ we cannot deduce Theorem \autoref{t0.1} by a reduction to orbit space $M/G.$ We do need a full-blown min-max argument.
	
	On the other hand, even though we use only one parameter here, the minimal surface we produce can have high index due to symmetry. For example, let $G=SO(2)\times SO(2)$ acting on $S^3(1)\s\R^4$ through $\rh_2\oplus\rh_2$, where $\rh_2$ is the natural rotation by $SO(2)$. The cohomogeneity is one, and the principal orbits are two dimensional Clifford tori, with one-dimensional exceptional orbits of circles. Thus, a min-max procedure will produce the minimal Clifford tours, which has index 5. For higher cohomogeneity, it is conceivable that one can even have higher indices.
	
	We have not used the Almgren-Pitts version of min-max in our construction due to technical reasons. The Almgren-Pitss theory is deeply rooted in the famous Almgren isomorphism theorem in \cite{FA},\begin{align}\label{isot}
	\pi_j(Z_k(M),0)\cong H_{j+k}(M,\Z)
	\end{align} However, if we're going to consider $G$-invariant cycles, then we have to first provide a suitable version of (\ref{isot}), of which we have not considered due to a lack of $G$-invariant homology theory. The referee has pointed out that using the double-cover argument of \cite{MN}, one can prove a codimension 1 modulo $2$ cycle version of \ref{isot}, i.e., the fundamental group of the space of $n$-dimensional $G$-invariant cycles mod $2$ is precisely $\Z/2\Z$. In principle, the regularity proof developed here should be able to carry over. It would be interesting to see what results in the non-invariant case, i.e. $G$ acting trivially, can carry over to this setting. For example, does there exist infinitely many $G$-invariant minimal hypersurfaces, if the cohomogeneity is larger than $2$?

	By using our method, we can prove that for cohomogeneity 2 actions, there always exist $G$-invariant minimal hypersurfaces which are regular outside of a singular set of Hausdorff dimension at most $n-1.$ In this case, the problem is essentially equivalent to finding geodesics with suitable boundary conditions on 2-dimensional orbifolds. Thus, we cannot rule out triple-junction like singularities, since such singularities appear in stable 1-dimensional varifolds. It would be interesting to use the existence of geodesics on 2-dimensional orbifolds to deduce the existence of invariant minimal hypersurfaces for cohomogeneity 2 actions. However, the question seems subtle as one cannot exclude the possibility that the geodesic will pass through the singular point of the orbifold \cite{CL}. To piece back the singular part of the geodesic requires a detailed analysis of the action itself. 
	
	We also note that if the action is free, then the theorem reduces to the one in \cite{DT}. The requirements of the connectedness of $G$ and the orientability of $M$ are both imposed for technical reasons and used essentially in the proof of Lemma \ref{free}.
	\subsection{Structure of the proof}
	In this section, we will sketch the main ideas and structure of our proof.
	
	We follow the usual convention to define cohomogeneity as the codimension of principal orbits with respect to $M$ (\cite{W}). 
	
	The following lemma says that in cohomogeneity $0$, the action is transitive and thus there doesn't exist any invariant hypersurface at all.
	\begin{lem}\label{l0.2}
		For a compact connected Lie group $G$ acting as isometries on $M^{n+1},$ the action is non-transitive if and only if the cohomogeneity is non-zero.
	\end{lem}
	\begin{proof}
		It is clear that cohomogeneity non-zero implies non-transitive. For the other direction, suppose the principal orbits are of the same dimension as $M^{n+1}$. Since orbits are a prior closed (\cite{W}), each of them must be the entire $M$, i.e, the action is transitive. 
	\end{proof}
	
	The cohomogeneity $1$ case is significantly easier and can be settled with an easy argument by using basic classifications as in \cite{PM}. We will deal with this in Section \ref{coho1}. The following arguments are for cohomogeneity at least $2.$
	
	The idea of the proof is as follows. The min-max construction in \cite{DT} can be broken down into four steps,
	
	Step 1. pulling-tight, which is the existence of stationary varifolds,\\
	Step 2. the existence of almost minimizing varifolds,\\
	Step 3. the existence of smooth stable replacements for almost minimizing varifolds,\\
	Step 4. regularity of varifolds with sufficiently many smooth stable replacements. 
	
	We will modify each step accordingly.
	
	First of all, our constructions are based on extracting time and again better-behaving subsequences of $G$-invariant sweepouts, so we have first to show the existence of such a thing. We will discuss this and fix our notations in Section \ref{term}
	
	In Section \ref{stv}, we will prove that stationary with respect to $G$-invariant vector fields implies stationary in general by using an averaging construction of Lie groups. Then we can adapt the pulling-tight procedure, which corresponds to Step 1.
	
	In Section \ref{amv}, we will run a modified combinatorial argument as in Section 3 of \cite{DT} to produce varifolds that are almost-minimizing among equivariant deformations, which corresponds to Step 2. 
	
	In Section \ref{repl}, we will use a modified argument of Section 4 of \cite{DT} to construct replacements with unknown regularity. To prove that replacements have the codimension $7$ regularity, which is essential to regularity theory, we will prove that these replacements are actually minimizing on a small enough scale with respect to all deformations. This is nontrivial since by construction they are expected to be only minimizing on a small scale with respect to $G$-invariant deformations. To this end, we use an argument first given by Lawson and Fleming in \cite{HL}. This is Step 3.
	
	Now that we have proved the existence of stable replacements with codimension $7$ regularity, we can simply reiterate and mimic the regularity theory developed in Section 5 of \cite{DT} to deduce the regularity of the $G$-invariant minimal hypersurface we have constructed. However, there are technical problems since we are no longer dealing with geodesic balls but tubes instead. Nevertheless, we will show that using the splitting of tangent cones proved in the appendix, the argument proceeds through. This is done in Section \ref{conc}.
	
	For an excellent introduction to basic Lie transformation group theory, please refer to \cite{W} and \cite{GB}. For our purposes, it is enough to know that there is a set of equivariantly diffeomorphic orbits of the highest dimension called principal orbits. The union of all such principal orbits forms an open dense set. The cohomogeneity is the codimension of any such principal orbit in $M.$ There are singular orbits that are smaller in dimension and exceptional orbits which are of the same dimension as the principal orbit and are topologically quotients of the principal orbits.
	
	Our paper is structured almost exactly the same as in \cite{DT}. To make comparisons easier, we intentionally try to phrase every theorem and definition as close as possible to \cite{DT} and will point out the difference explicitly. For a beautiful discussion of the framework of proof and amazingly informative figures, we strongly urge the reader to read \cite{DT}.
	\section{Terminology}\label{term}
	This section corresponds to Section 0 and 1 of \cite{DT}. We will fix the terminologies and prove the basic existence of a nontrivial family of $G$-invariant sweepouts.
	\subsection{Notations}
	Since we are constantly dealing with $G$-invariant objects in our paper, we will sometimes add $G$- in front of objects to indicate they are $G$-invariant. The exact definition of $G$-invariance is almost always clear from the context. We will list some here.  
	\begin{itemize}[noitemsep]
		\item[] A $G$-varifold $V$ satisfies $g\pf V=V$ for any $g\in G.$ 
		\item[] A $G$-vector field $X$ satisfies $g\pf X=X$ for all $g\in G.$ 
		\item[] A $G$-isotopy $\Phi(t)$ satisfies $g\m\circ \Phi(t)\circ g=\Phi(t)$ for all $t$ and $g\in G.$
		\item[] A $G$-set ($G$-neighborhood) is an (open) set which is a union of orbits.
	\end{itemize}
	Many other notations are the same as those in the paper \cite{DT}. However, we will sometimes add a subscript or superscript $G$ to signify $G$-invariance. We will summarize those as follows. 
	\newcommand{\an}{\textnormal{An}^G}
	\newcommand{\ann}{\mathcal{AN}^G}
	\begin{description}[font=\sffamily\bfseries,leftmargin=3cm, style=nextline]
		\item[$\pi_G$] the projection $\pi_G:M\mapsto M/G$ defined by $x\mapsto [x].$
		\item[$B_\rh^G(x),\ov{B}_\rh^G(x)$]
		open and closed tubes with radius $\rh$ around the orbit $G.x$
		\item[$\mathfrak{X}^G(M)$] the space of $G$-vector fields on $M$.
		\item[$\an(x,t,\tau)$] the open tube $B_t^G(x)\sm \ov{B}^G_\tau(x)$.
		\item[$\ann_r(x)$] the set $\{\an(x,\tau,t)|0<\tau<t<r\}.$
		\item[$d_G(U,V)$]  $\displaystyle\inf_{\substack{x\in \pi_G(U),\\y\in\pi_G(V)}}d_{M/G}(x,y),$ with $d_{M/G}$ the induced distance on $M/G$.
		\item[$\diam_G U$] the diameter of the projection $\pi_G(U)$ of a $G$-set $U$ in the metric space $M/G.$
		\item[$T_yG.x$] the tangent space of the orbit $G.x$ at some point $y$ on $G.x.$
	\end{description}
	Note that all of these are well defined in that $M/G$ is also a complete Hausdorff metric space (\cite{W}).
	\subsection{Basic definitions}
	In what follows, $M$ will denote a compact $(n+1)$ dimensional smooth Riemannian manifold without boundary. Let $G$ be a compact connected Lie group acting smoothly as isometries on $M$, with bi-invariant Haar measure $\mu$ on $G$ normalized to $\mu(G)=1.$
	
	First of all, following the convention of \cite{DT}, we will fix what we mean by a generalized hypersurface.
	\begin{defn}\label{d1.5}
		(Definition 1.5 in \cite{DT})	A generalized hypersurface $\Ga\s U$ is an $n$-dimensional integral varifold whose support is of Hausdorff dimension at most $n$ so that $\Ga$ is smooth outside a set $\textnormal{Sing}\Ga$ of Hausdorff dimension at most $n-7.$ 
	\end{defn}
	We will add stable (minimal) in front of generalized hypersurface to mean it is stable (stationary, respectively) as a varifold.
	
	The notion of generalized smooth family is just adding $G-$ to objects in \cite{DT}, but we will need a weaker notion of sweepout.
	\begin{defn}\label{swpt}(Compare Definition 0.2 in \cite{DT}.)
		A $G$-generalized smooth family is a $k$-parameter family of generalized hypersurfaces $\{\Ga_t\}_{t\in[0,1]^k}$ with the following properties
		\begin{itemize}[nosep]
			\item[(s0)] $\hm^n(\Ga_t)<\infty$ for all $t\in [0,1]^k$
			\item[(s1)] For all $t\in[0,1]^k$, $\Ga_t$ is a $G$-invariant smooth hypersurface in $\Ga_t\sm P_t,$ where $P_t$ consists of finitely many disjoint orbits;
			\item[(s2)] $\hm^n(\Ga_t)$ is continuous in $t$ for $t\in[0,1]^k.$ Moreover, $t\mapsto \Ga_t$ is continuous if we use Hausdorff topology on subsets in $M$;
			\item[(s3)] For any $U\s\s M\setminus P_{t_0}$, we have $\Ga_t\to \Ga_{t_0}$ smoothly as $t\to t_0.$
		\end{itemize}
		Moreover, a $G$-generalized family $\{\Ga_t\}_{t\in[0,1]}$ is a called $G$-sweepout of $M$ if there exists a one-parameter family of $G$-invariant open sets $\{\Om_t\}_{t\in[0,1]}$ satisfying
		\begin{itemize}[nosep]
			\item[(sw1)] $\Ga_t\sm\pd\Om_t\s P_t$ for all $t$.
			\item[(sw2)] $\Om_0=\es,\Om_1=M;$
			\item[(sw3)] $\vol(\Om_t\sm\Om_s)+\vol(\Om_s\sm\Om_t)\to 0$ as $t\to s$;
		\end{itemize}
	\end{defn}
	The only difference between our definition and Definition 0.2 in \cite{DT} is we define $P_t$ to be a finite set consisting of orbits, instead of points. We adopt this notion because passing through a critical point of an equivariant Morse function amounts to adding a handle bundle as in \cite{AW} instead of adding cells as in the usual Morse theory. Thus in general we have to assume sweepouts start and end at orbits. However, this change will not hinder our proof. By equivariant Morse theory as in \cite{AW}, equivariant Morse functions are dense in the space of smooth $G$-invariant functions, which come in abundance by lifting smooth functions on the quotient space. Moreover, by taking the gradient of those functions, we deduce the existence of nontrivial $G$-vector fields.
	\begin{prop}\label{le}
		If $f:M\to [0,1]$ is a $G$-equivariant Morse function in the sense of \cite{AW}, then $\{\Ga_t=f\m(t)\}_{t\in[0,1]}$ is a $G$-sweepout.
	\end{prop}
	\begin{proof}The proof is the same as proving level sets of Morse function form a sweep-out. The only part that might need attention is to prove that there are only finitely many orbits that might form non-smooth parts of critical submanifolds. This can be deduced by using Lemma 4.1 in \cite{AW}, the equivariant Morse lemma.
	\end{proof}
	
	For any one-parameter generalized family $\{\Ga_t\},$ we define
	\begin{align*}
	\F(\{\Ga_t\})=\max_{t\in[0,1]}\hm^n(\Ga_t).
	\end{align*}
	\begin{prop}\label{pst}
		$\F(\{\Ga_t\})\ge C(M)$ for any sweepout $\{\Ga_t\}$, where $C(M)$ is a positive constant depending only on $M.$
	\end{prop}
	\begin{proof}
		The same proof of Proposition 0.5 in \cite{DT} carries over, which only uses isoperimetric inequality.
	\end{proof}
	For a family $\Lambda$ of sweepouts, let
	\begin{align*}
	m_0(\Lambda)=\inf_{\{\Ga\}_t\in\Lambda}\F=\inf_{\{\Ga_t\}\in \Lambda}\max_{t\in [0,1]}\hm^n(\Ga_t).
	\end{align*}
	By Proposition \ref{pst}, $m_0(\Lambda)\ge C(M)>0.$ We call a sequence $\{\{\Ga_t\}^k\}\s \Lambda$ minimizing if
	\begin{align*}
	\lim_{k\to\infty}\F(\{\Ga_t\}^k)=m_0(\Lambda).
	\end{align*}
	A sequence of hypersurfaces $\{\Ga^k_{t_k}\}$ is called a min-max sequence of a family $\Lambda$ if $\{\{\Ga_t\}^k\}$ is minimizing and $\lim_k\hm^n(\Ga_{t_k}^k)=m_0(\Lambda).$ 
	
	We will only deal with families $\Lambda$ closed under the following notion of homotopy.
	\begin{defn}(Compare Definition 0.6 in \cite{DT}.)
		We call two sweepouts $\{\Ga_s^0\}$ and $\{\Ga_s^1\}$ homotopic if for some two parameter $G$-generalized family $\{\Ga_t\}_{t\in[0,1]^2}$ we have $\Ga_{(0,s)}=\Ga_s^0$ and $\Ga_{(1,s)}=\Ga_s^1$. A family $\Lambda$ of $G$-sweepouts is said to be $G$-homotopically closed if $\{\Ga_t\}\in \Lambda$ implies that any sweepout homotopic to $\{\Ga_t\}$ is contained in $\Lambda.$
	\end{defn}
	The following smaller classes of $G$-homotopies will also be very useful.
	\begin{defn}(Compare Definition 2.1 in \cite{DT}.)
		Let $X:[0,1]\to \mathfrak{X}(M)$ be a smooth map to the space of smooth vector fields on $M.$ Suppose $F([0,1])\s \mathfrak{X}^G(M).$ Let $\Psi_t:[0,1]\times M\to M$ be the diffeomorphism corresponding to $F(t)$. If $\{\Ga_t\}_{t\in[0,1]}$ is a $G$-sweepout, then $\{\Psi_t(s,\Ga_t)\}_{(t,s)\in[0,1]^2}$ is called a $G$-homotopy from $\{\Ga_t\}$ to $\{\Psi_t(1,\Ga_t)\}$ induced by ambient isotopies.
	\end{defn}
	And finally, we will give the definition of $G$-almost minimizing varifolds, which is essentail to our regularity theory,
	\begin{defn}
		Fix $\e>0$ and a open $G$-set $U\s M$. Suppose $\Om$ is another open $G$-set. Then the boundary $\pd \Om$ of $G$-open set in $M$ is $\e$-$G$-almost minimizing ($\e$-G-a.m.) in $U$ if there are no 1-parameter families of boundaries of open $G$-sets $\Om_t,t\in{[0,1]},$ so that
		\begin{itemize}[nosep]
			\item[(a.1)] (s0) (s1), (s2), (s3), (sw1), and (sw3) of Definition \ref{swpt} hold,
			\item[(a.2)] $\Om_0=\Om,$ and $\Om_t\setminus U=\Om\setminus U$ for every $t$,
			\item[(a.3)]$\hm^n(\pd \Om_t)\le \hm^n(\pd \Om)+\frac{\e}{8}$ for all $t\in[0,1],$
			\item[(a.4)] $\hm^n(\pd\Om_1)\le \hm^n(\pd \Om)-\e.$	
		\end{itemize}
		If there exists a sequence $\e_k\to 0$ so that a collection $\{\pd \Om^k\}$ of generalized hypersurfaces is $\e_k$-G-a.m. in $U$, then we say $\{\pd \Om^k\}$ is $G$-almost minimizing in $U.$ Note that if $V\s U$ are both $G$-open sets, then an $\e$-$G$-a.m. set in $U$ is also $\e$-$G$-a.m. in $V.$
	\end{defn}
	One major difference between almost minimizing in \cite{DT} and $G$-almost minimizing is that for $G$-a.m. we're only considering deformations under $G$-vector fields. In fact, this difference is significant and cannot be remedied easily, unlike the distinction between $G$-stationary and stationary in the next section.
	
	Finally, we need the notion of replacement. The definition is the same as \cite{DT}, and we do not impose that stability is with respect to $G$-invariant vector fields only. The replacements we construct are a priori only stable with respect to invariant vector fields, but a simple first eigenfunction argument proves stability with respect to all vector fields.
	\begin{defn}
		(Definition 2.5 in \cite{DT}) Let $V\in\mathcal{V}(M)$ be a stationary varifold and $U\s M$ be an open set. A stationary varifold $V'\in \mathcal{V}(M)$ is called a replacement for $V$ in $U$ if $V'=V$ on $M\sm\ov{U}$, $\no{V'}(M)=\no{V}(M)$ and $V'\ip U$ is a stable minimal generalized hypersurface.
	\end{defn}
	\section{Existence of $G$-invariant Stationary Varifolds}\label{stv}
	This section will be dedicated to proving the following proposition
	\begin{prop}\label{esv}(Compare Proposition 2.2 in \cite{DT}.)
		If $\Lambda$ is a family of $G$-sweepouts closed under $G$-homotopies induced by ambient $G$-isotopies, then there exists a minimizing sequence $\{\{\Ga_t\}^k\}\s\Lambda$ so that if $\{\Ga_{t_k}^k\}$ is a min-max sequence, then $\Ga_{t_k}^k\to V$ for some stationary $G$-varifold $V $.
	\end{prop}
	In \cite{DT}, the proof refers to Proposition 4.1 of \cite{CD}. In our case, we first have to develop some basic facts about stationary properties under $G$-vector fields, and then the same proof as in \cite{DT} applies.
	\subsection{$G$-stationary implies stationary}
	The idea of development in this subsection is inspired by Section 3 in \cite{DK}. By abuse of notation, we use $\no{\de V}_G(O)$ to denote the total first variation with respect to $G$-vector fields compactly supported on an open set $O$, i.e.,
	\begin{align*}
	\no{\de V}_G(O)=\sup\{\de V(\chi)|g\pf \chi=\chi,\forall g\in G,\no{\chi}\le 1,\text{spt}\chi\s O\}.
	\end{align*}
	We will use $\no{\de V}_G$ to  denote $\no{\de V}_G(M).$
	\begin{defn}
		A $G$-varifold $V$ is $G$-stationary if $\no{\de V}_G(M)=0.$
	\end{defn}
	\begin{lem}\label{cute}
		For any $G$-varifold $V$, and $G$-neighborhood $O$, we have
		\begin{align*}
		\no{\de V}_G(O)=\no{\de V}(O).
		\end{align*}
		
	\end{lem}
	\begin{proof}
		It suffices to prove that for any vector field $X,$ with $|X|\le 1$ supported in $O,$ there exists a $G$-vector field $X_G$, with $|X_G|\le 1$ so that,
		\begin{align*}
		\de V(X)(O)=\de V(X_G)(O).
		\end{align*}
		Use $\psi(t)$ to denote the diffeomorphisms generated by $X.$ Consider the modified diffeomorphism
		\begin{align*}
		\psi_g(t)=g\m \circ\psi(t)\circ g.
		\end{align*}
		Let $X_g$ to be the vector fields corresponding to $\psi_g(t).$ Now define
		\begin{align*}
		X_G(p)=\int_G X_g(p)d\mu(g),
		\end{align*}where the integral is carried out in $T_pM.$ By construction, $X_G$ is supported in $O.$ We have
		\begin{align*}
		g\pf X_G(p)=&g\pf\int_G  X_h(g\m p)d\mu(h)
		=\int_G g\pf X_h(g\m p)d\mu(h)\\
		=&\int_G \frac{d}{dt}g\circ h\m \psi(t)\circ h\circ g\m d\mu(h)	=\int_G X_{hg\m}(p)d\mu(h)\\
		=&\int_G X_{h}(p)d\mu(h)=X_G(p),
		\end{align*} so $X_G$ is $G$-invariant. 
		Since $g$ and $g\m$ are all isometries, we have $(g\m\circ\psi(t)\circ g)\pf V=g\m\pf(\psi(t)\pf( g\pf V))=\psi(t)\pf V.$ By linearity of first variation, we can conclude that
		\begin{align*}
		\de V(X_G)(O)=&\int_{G_n(O)} \di_S X_GdV(x,S)	=\int_{G_n(O)} \di_S \int_G X_gd\mu(g)dV(x,S)\\
		=&\int_G \int_{G_n(O)} \di_SX_gdV(x,S)d\mu(g)
		=\int_G\de V(X_g)(O)d\mu(g)\\
		=&\int_G \de V(X)(O)d\mu(g)
		=\de V(X)(O).
		\end{align*}by Fubini theorem. Finally, note that
		\begin{align*}
		|X_G|^2=&\ri{X_G,\int_GX_hd\mu}
		=\int_G\ri{X_G,X_h}d\mu
		\le\int_G|{X_G}||X_h|d\mu(h)
		=|X_G|\int|X_h|d\mu(h),
		\end{align*}
		so this yields
		\begin{align*}
		|X_G|\le {\int_G|X_h|d\mu}=1.
		\end{align*}	
	\end{proof}
	\begin{cor}\label{gst}
		A $G$-stationary $G$-varifold $V$ is stationary.	\end{cor}
	\begin{proof}
		By letting $O=M$ in Lemma \autoref{cute}, we deduce immediately the desired result.
	\end{proof}
	\subsection{Proof of Proposition \ref{esv}}\label{pfe}
	Let $\{\{\Si_t\}^n\}$ be a minimizing sequence. We will deform it into another sequence $\{\Ga_t^n\}$ using ambient $G$-isotopies so that any min-max subsequence of $\{\{\Ga_t\}^n\}$ converges to a stationary varifold. 
	
	First, let's consider the varifolds with mass bounded by $4m_0,$ and call the collection $X.$ Metrize it with weak-$\du$ topology induced by Riesz representation. Now, let $X^G$ be the subspace of $G$-varifolds in $X.$ $X^G$ is closed by construction and thus a compact subset of $X.$ Let $\mv_\infty^G=X^G\cap \sv$ denote the space of $G$-stationary varifolds, which is closed by construction. By our lemma, it is a subset of the set of stationary varifolds $\mv_\infty$. Thus, the distance to $\sv^G$ is a well-defined continuous function on $X$ and thus $X^G$. Now, we can consider the annuli
	\begin{align*}
	\mv_k=\{V\in X^G|2^{-k-1}\le \vd(V,\sv^G) \le 2^{-k+1}\}.
	\end{align*} 
	Then the proof proceeds almost the same as the proof of Proposition 4.1 in \cite{CD}, by adding the appropriate superscript $G$ to denote invariance. In essence, we just replace $\sv$ with $\sv^G,$ $G$-invariant stationary varifolds, and let all the vector fields used in construction to be $G$-invariant. For the partition of unity, we note that it still holds in the closed subspace of $G$-invariant varifolds by Theorem II.2 in \cite{RH}. 
	\section{Existence of $G$-almost minimizing varifolds in $G$-annuli}\label{amv}
	In this section, we will prove the following proposition.
	\begin{prop}\label{exam}(Compare Proposition 2.4 in \cite{DT}.)
		Suppose $\Lambda$ is a family of $G$-sweepouts closed under $G$-homotopies. Then there exists a $G$-invariant function $r:M\to\R_+$ and a min-max sequence $\Ga^k=\Ga_{t_k}^k$ so that
		\begin{itemize}[nosep]
			\item[(1)] $\{\Ga^k\}$ is $G$-a.m. in every $\an\in\ann_{r(x)}(x),x\in M$
			\item[(2)] $\Ga^k\to V$ as $k\to\infty $ for some stationary $G$-varifold $V$
		\end{itemize}
	\end{prop} 
	This idea of the proof is the same as Section 3 of \cite{DT}. However, we will need to make some technical amendments.
	\subsection{$G$-almost minimizing varifolds}
	Before coming to the proof, we define the basic notions needed.
	\begin{defn}(Compare Definition 3.2 in \cite{DT}.)
		For two $G$-open sets $U^1,U^2$, a $G$-generalized hypersurface is said to be $\e$-$G$-a.m. in $(U^1,U^2)$ if it is $\e$-$G$-a.m. in at least one of the two open sets. We define $\co$ to be the collection of pairs $(U^1,U^2)$ of $G$-open sets with 
		\begin{align*}
		d_G(U^1,U^2)\ge 4\min\{\diam_G( U^1),\diam_G( U^2)\}.
		\end{align*}
	\end{defn}
	This definition differs from Definition 3.2 in \cite{DT} in that we consider both diameter and distance on the quotient $M/G$ instead of on $M.$ This shift is essential because otherwise there might be too few sets in $\co.$ We owe this idea to \cite{DK}. Also, $d_G$ is not Hausdorff distance on closed sets. It is just the infimum of the distance between pairs of points in the corresponding sets.
	
	The following simple lemma utilizing only the metric space property will be of great importance later.
	\begin{lem}(Compare Lemma 3.3 in \cite{DT}.)
		If $(U^1,U^2)$ and $(V^1,V^2)$ satisfy
		\begin{align*}
		d_G(U^1,U^2)\ge 2\min\{\diam_G(U^1),\diam_G(U^2)\},\\
		d_G(V^1,V^2)\ge 2\min\{\diam \pi_G(V^1),\diam_G( V^2)\},
		\end{align*}
		then there exist $i,j\in\{1,2\}$ so that $d_G(U^i,V^j)>0$ and thus $d>0$ if $U^i,V^j$ are $G$-sets.
	\end{lem}
	\begin{proof}
		Since $M$ is compact, we can find $x^i_j\in\ov{\pi_G(U^i)},y^j_i\in \ov{\pi_G(V^j)}$ realizing the distance $d_G(U^i,V^j)$. Note that $d_G(x^i_{j_1},x^i_{j_2})\le \diam_G(U^i)$ since the two points lie in $\ov{\pi_G(U^i)}.$ Suppose the conclusion doesn't hold, then $d_G(U^i,V^j)=0$ for all $i,j$. Thus, we can choose $x^i_j$ to coincide with $y^j_i.$ Without loss of generality, suppose $\diam_G(U^1)\le \diam \pi_G(U^2)$ and $\diam \pi_G(V^1)\le \diam \pi_G(V^2).$ This gives\begin{align*}
		\diam_G(U^1)&\ge d_G(x^1_1,x^1_2)=d_G(y^1_1,y^2_1)\ge d_G(V^1,V^2)\\&\ge 2 \diam \pi_G(V^1)\ge 2d_G(y^1_1,y^1_2)= 2d_G(x^1_1,x^2_1)\\&\ge 2d_G(U^1,U^2)\ge 4\diam_G(U^1),
		\end{align*} which is a contradiction.
	\end{proof}
	The most essential ingredient for the proof of Proposition \ref{exam} is the following Almgren-Pitts combinatorial lemma.
	\begin{prop}\label{apc}
		(Almgren-Pitts combinatorial lemma) Let $\Lambda$ be a $G$-homotopically closed family of $G$-sweepouts. There exists a min-max sequence $\{\Ga^N\}=\{\pd\Om_{t_k(N)}^{k(N)}\}$ so that
		\begin{itemize}[nosep]
			\item[(1)] 	$\Ga^N$ converges to a stationary varifold;
			\item[(2)]	For any $(U^1,U^2)\in\co,$ $\Ga^N$ is $1/N$-$G$-a.m. in $(U^1,U^2)$ for $N>N(U^1,U^2)$ large enough, with $N(U^1,U^2)>0$ depending on $(U^1,U^2).$
		\end{itemize}
	\end{prop}
	Proof of Proposition \ref{apc} is exactly the same as proof of Proposition 3.4 in \cite{DT}. We only need to substitute Lemma \ref{free} below for Lemma 3.1 in \cite{DT} and adding $G$- in front of objects. We will omit the proof.
	
	\begin{proof}[Proof of Proposition \ref{exam}]
		The proof has essentially the same idea. We show that a subsequence of the $\{\Ga^k\}$ in Proposition \ref{apc} satisfies the requirements of Proposition \ref{exam}. 
		
		By the existence of equivariant tubular neighborhoods \cite{GB}, for any $z\in M,$ there exists a nonzero $\rh_G(z)$ so that for all $0<\rh\le \rh_G(z)$, $B^G_{\rh}(z)$ is a well-defined $G$-invariant tubular neighborhood around $x.$ Note that we cannot have uniform lower bound on $\rh_G(z)$. For example, if there exists a fixed point $z^\ast$ of $G$, then for points closer and closer to of $z^\ast,$ the allowable radii of invariant tubes have to be smaller and smaller. However, apparently we can have a uniform upper bound on $\rh_G$ by injectivity radius. 
		
		For any $x\in M$, we fix $k\in\N$ and some choice of radius $0<\rh(x)<\frac{1}{9}\rh_G(x)$. (The exact choice of $\rh(x)$ does not matter. We only need it to be positive. Moreover, it can be made $G$-invariant by pushing forward along orbits.) For all $x\in M$, we have $(B_\rh^G,M\sm\ov B_{9\rh}^G(x))\in\co$ by construction. By Proposition \ref{apc}, for $k$ large, $\Ga^k$ is $1/k$-$G$-almost minimizing in either $B^G_\rh(x)$ or $M\sm \ov{B}_{9\rh}^G(x).$ Consequently, for our choice $\rh(x)>0,$ we have
		\begin{itemize}[nosep]
			\item[(a)] either $\{\Ga^k\}$ is $1/k$-$G$-a.m. in $B^G_{\rh(y)}(y)$ for every $y\in M.$
			\item[(b)] or there exists a subsequence $\{\Ga^k\}$ (not relabeled) and a sequence $\{x^k\}\s M$ such that $\Ga^k$ is $1/k$-$G$-a.m. in $M\sm\ov{B}_{9\rh(x^k)}^G(x^{k})$.
			
		\end{itemize}
		If for some choice of radius $\ai\rh(x)>0$ with $\ai\in(0,1],$ (a) holds, then we're fine. If this is not the case, then we can find a subsequence of $\{\Ga^k\}$ (not relabeled) and a collection of points $\{x_j^k\}_{j,k\in \N_+}\s M$ so that
		\begin{itemize}[nosep]
			\item[(i)] for any fixed $j,$ $\Ga^k$ is $1/k$-$G$-a.m. in $M\sm\ov{B}^G_{\rh(x_j^k)/j}(x_j^k)$ for $k$ large enough,
			\item[(ii)] $x_j^k\xrightarrow{d_G} x_j$  for $k\to\infty,$ i.e., $G.x_j^k$ converges to $G.x_j$ in the quotient space $M/G,$ and $x_j\xrightarrow{d_G} x$ for $j\to\infty.$
			
		\end{itemize}
		\begin{claim}
			($\ast$) For any $J>\frac{1}{\rh(x)}$, there exists $K_J$ so that $\Ga^k$ is $G$-$1/k$-a.m. in $M\setminus\ov{B}^G_{1/J}(x)$ for all $k\ge K_J.$ 
		\end{claim}This can be done by choosing $j$ with $d_G(x_j,x)<1/(3J),$ and more importantly
		$$\sup_{z\in M}\rh^G(z)/j\le\frac{1}{3J}.$$ Then take $k$ large enough with $d_G(x_j^k,x_j)<1/(3J)$ and $\Ga^k$ $1/k$-$G$-a.m. in $M\sm\ov{B}_{\rh(x_j^k)/j}^G(x_j^k).$ Note that $$\frac{\rh(x_j^k)}{j}+d_G(x_j^k,x_j)+d_G(x_j,x)<\frac{1}{J},$$ so we have $M\sm \ov{B}^G_{1/J}(x)\s  M\sm\ov{B}_{\rh(x_j^k)/j}^G(x_j^k).$ This proves the claim. Thus, for $y\in M\sm\{x\}$, we can simply choose $r(y)<\rh^G(y)$ so that $B^G_{r(y)}\s\s M\sm\{x\}$. By construction we have that $\an_{r(z)}\s\s M\sm\{x\}$ for any $\an\in \ann_{r(z)}(z)$ with $z\in M\sm \{x\} $. By ($\ast$), this definition of $r(y)$ satisfies the requirements in the proposition. This defines $r$ for $M\sm\{x\}$. For $x$ itself, note that as long as $r(x)<\rh(x),$ then $\Ga^k$ would be $1/k$-$G$-a.m. for $k$ large enough in any annulus around $x$ by ($\ast$), since the annulus will be contained in a complement of an invariant tube. 
	\end{proof}
	For the proof of Proposition \ref{apc}, we need an important lemma that will help us construct dynamic competitors and glue them to get contradictions in our omitted proof of Proposition \ref{apc}
	\begin{lem}\label{free}(Compare Lemma 3.1 in \cite{DT}.)
		Let $U\s\s\ U'\s M$ be two $G$-open sets and $\{\pd \Xi_t\}_{t\in[0,1]}$ be a $G$-sweepout. For $\e>0$, and $t_0\in [0,1]$, assume $\{\pd \Om_s\}_{s\in [0,1]}$ is a one-parameter family of $G$-generalized hypersurfaces satisfying (a.1), (a.2), (a.3), and (a.4), with $\Om=\Xi_{t_0}$. Then there is $\eta>0$ such that the following holds for every $a,b,a',b'$ with $t_0-\eta\le a<a'<b'<b\le t_0+\eta.$ 
		
		We can find a competitor $G$-sweepout $\{\pd \Xi'_t\}_{t\in [0,1]}$ so that
		\begin{itemize}[nosep]
			\item[(a)] $\Xi_t=\Xi'_t$ for $t\in [0,a]\cup [b,1]$ and $\Xi_t\sm U'=\Si'_t\sm U'$ for $t\in (a,b)$;
			\item[(b)] $\hm^n(\pd \Xi'_t)\le \hm^n(\pd\Xi_t)+\frac{\e}{4},$ for every $t$;
			\item[(c)] $\hm^n(\pd \Xi'_g)\le \hm^n(\pd \Xi_t)-\frac{\e}{2}$ for $t\in(a',b').$
			\item[(d)] $\pd \Xi'_t$ is $G$-homotopic to $\{\pd \Xi_t\}.$
		\end{itemize}
	\end{lem}
	\begin{proof}
		The proof is the same as proof of Lemma 3.3 in \cite{DT}. There are several points worth mentioning. First, by \cite{W}, we can find invariant partition of unity subordinate to any $G$-open set. Second, when we fix normal coordinates $(z,\si)\in\pd\Xi_{t_0}\cap C\times (-\de,\de),$ we are actually identifying the trivial normal bundle as the coordinates. In other words, we identify a $G$-invariant tubular neighborhood of $\pd\Xi_{t_0}\cap C$ with the trivial normal bundle of $\Xi_{t_0}.$ This is possible for the following reasons. All boundaries we consider are two-sided and thus naturally orientable with trivial normal bundle. Thus, we can choose a unit normal field $\nu$ well-defined except at finitely many orbits. Since $G$ acts by isometries, $g\pf\nu_p=\pm\nu_{g.p}$. Those $g$ reversing the normal will automatically form an index $2$ subgroup of $G$, which is contradictory to connectedness of $G.$ Thus, $G$ preserves the normal of $\pd\Xi_{t_0}\cap C.$ We can deduce that $G$-vector fields on $\pd\Xi_{t_0}\cap C$ can be identified with $G$-smooth functions on $\pd\Xi_{t_0}\cap C$. Moreover since the exponential map is equivariant, we see that exponentiating any $G$-vector field in the normal bundle with small enough norm would yield a $G$-invariant generalized hypersurface. Using these facts above, the $G$-invariance of our constructions can be readily verified. This index-2 subgroup argument is inspired by a conversation with Professor Robert Bryant.
	\end{proof}
	
	\section{The existence of $G$-invariant replacements}\label{repl}
	This section is dedicated to the proof of the following proposition
	\begin{prop}\label{p2.6}(Compare Proposition 2.6 in \cite{DT}.)
		Let $\{\Ga^j\}$, $V$ and $r$ be as in Proposition \ref{exam}. Fix $x\in M$ and consider an annulus $\an\in\ann_{r(x)}(x).$ Then there exists a $G$-varifold $\ti{V}$, a $G$-sequence $\{\ti{\Ga}^j\}$ and a $G$-function $r':M\to\R_+$ such that
		\begin{itemize}[nosep]
			\item[(a)] $\ti{V}$ is a replacement for $V$ in $\an$ and $\ti{\Ga}^j$ converges to $\ti{V}$ in the sense of varifolds;
			\item[(b)] $\ti{\Ga}^j$ is $G$-a.m. in every $\textnormal{An}^G_{'}\in \ann_{r'(y)}(y)$ with $y\in M$,
			\item[(c)] $r'(x)=r(x).$
		\end{itemize}
	\end{prop}
	\begin{proof}
		Assume Lemma \ref{l4.1} and \ref{l4.2} below. Then exactly the same proof in Section 4.4 in \cite{DT} would carry over. The only cautious point is arguing that $\ti{V}$ is stationary. Using the same argument, we can only deduce that it is stationary among $G$-invariant varifolds. Invoke Corollary \ref{gst} to deduce that $\ti{V}$ is in fact stationary.
	\end{proof}
	This proposition is the basis on which we can bootstrap and utilize to prove the regularity of varifolds with good replacements. Our definition of replacements is exactly the same as Definition 2.5 in \cite{DT}. We do not require any $G$-invariance in the definition of the replacements. Instead, though our replacements are $G$-invariant by construction, and they will be stable with respect to all deformations by a simple first eigenfunction argument.
	
	Now, we fix some $\an\in\ann_{r(x)}{(x)}.$
	\subsection{Setting}
	For every $j,$ consider the class $\hm(\Om^j,\an)$ of $G$-sets $\Xi$ such that there exists a family $\{\Om_t\}$ satisfying $\Om_0=\Om^j,$ $\Om_1=\Xi,$ (a.1), (a.2), (a.3), for $\e=\frac{1}{j}$ and $U=\an.$ Now, pick a sequence $\Ga^{j,k}=\pd \Om^{j,k}$ which is minimizing for the perimeter in the class $\hm(\Om^j,\an)$. Up to subsequences, we can assume that
	\begin{itemize}[nosep]
		\item[]$\Om^{j,k}$ converges to a Caccioppoli $G$-set $\ti{\Om}^j$,
		\item[]	$\Ga^{j,k}$ converges to a $G$-varifold $V^j$;
		\item[]	$V^j$ (and a suitable diagonal sequence $\ti{\Ga}^j$=$\Ga^{j,k(j)}$) converges to a $G$-varifold $\ti{V}$.	
	\end{itemize}
	All the convergence comes from basic compactness theorem for integral currents and varifolds, and the equivalence of Cacciopoli sets with codimension-$1$ integral currents of finite mass. The $G$-invariance comes from the fact that $G$-invariant objects will form a closed subspace in both of these two cases.
	
	The proof of Proposition \ref{p2.6} will be broken into three steps. First, we need to prove the following lemma for the regularity of the minimizers $\ti{\Om}^j.$
	\begin{lem}\label{l4.1}(Compare Lemma 4.1 in \cite{DT}.)
		For every $j$ and $y\in \an,$ there exists a $G$-tube $B=B_\rh^G(y)\s \an$ and some $k_0\in\N$ with the following property. Every open $G$-set $\Xi$ such that $\pd\Xi$ is smooth except for a finite union of orbits,	$\Xi\sm B=\Om^{j,k}\sm B$, and $\hm^n(\pd\Xi)<\hm^n(\pd \Om^{j,k})$, is contained in the collection $\hm(\Om^j,\an)$ for $k\ge k_0.$
	\end{lem}
	Using the above lemma, we would like to show that
	\begin{lem}\label{l4.2}(Compare Lemma 4.2 in \cite{DT}.)
		$\pd\ti{\Om}^j\cap \an$ is a stable minimal generalized hypersurface in $\an$ and $V^j\lrcorner\an=\pd\ti{\Om}^j\lrcorner	 \an$.
	\end{lem}
	However, for proof of Lemma \ref{l4.2}, we have to work a little bit harder than the proof of Lemma 4.2 in \cite{DT}. The idea of that proof can be utilized, but since every object and deformation are $G$-invariant, we cannot use the regularity for Plateau problem. Instead, we have to work harder for a regularity result for equivariant Plateau problem in our setting.
	\subsection{Proof of Lemma \ref{l4.1}}
	Step 1 in Section 4.2 of \cite{DT} can be used unchanged. Transversality in the proof can be deduced from Theorem 6.35 (Parametric Transversality) in \cite{JL}, since $\Ga^{j,k}$ is smooth except for finitely many orbits, which corresponds to finitely many values of radius in the tube. The constructions of cones is a little different. For each $z\in\ov{B}^G_r(y)$, there is a unique geodesic from $z$ to $G.y$ and intersecting $G.y$ orthogonally by construction of tubular neighborhood. Denote this geodesic $[G.y,z].$ As usual, $(G.y,z)=[G.y,z]\sm (G.x\cup\{z\})$. We let $K$ be the open cone,
	\begin{align*}
	K=\bigcup_{z\in\pd B^G_r(y)\cap \Om^{j,k}}(G.y,z).
	\end{align*}By construction $K$ is $G$-invariant and smooth. If we use $\exp^\perp$ to denote the exponential map of the normal bundle $N(G.y)$ of $G.y$ in $M,$ then $\exp^\perp$ is a diffeomorphism on the tube $B_r^G(y).$ Now, if $K'$ is the cone over $(\exp^\perp)\m \Om^{j,k}\cap \pd B_r(y)$ in the normal bundle, then $\exp^\perp K'=K.$ (Cone over a generalized hypersurface $S$ in the normal bundle is $C(S)=\{(p,tv)|t\in[0,1],(p,v)\in S\}$)
	The rest of Step 1 in Section 4.2 of \cite{DT} can be adapted easily.
	
	Step 2 of the proof in \cite{DT} is the volume estimates, which mostly carry over unchanged and consists of basic calculus on manifolds combined with tubular neighborhoods. We replace every geodesic balls in those estimates with invariant tubes instead. Note that inequality (4.5) in \cite{DT} shall have $\hm^n(\pd K)$ on the left hand side.
	
	The only part that needs caution is the monotonicity formula estimate (4.13) in \cite{DT}. By the convergence of varifolds, we still have $$\hm^n(\pd \Om^{j,k}\cap B_{2\rh}^G(y))\le 2\no{V_j}(B_{4\rh}^G(y)).$$
	However, recall that $B_{\rh}^G(y)$ is a tube around $G.y,$ so we ca not use the monotonicity formula immediately. 
	
	Now, invoke lemma \ref{la1} by choosing $20\rh<\rh_0.$ By $G$-invariance of $V^j,$ and the monotonicity formula applied to $V^j,$ we can deduce that
	\begin{align*}
	\no{V^j}(B^G_{4\rh}(y))&\le \no{V^j}\left(\bigcup_{z\in \mathcal{B}}B_{20\rh}(z)\right)\\
	&\le |\mathcal{B}|\no{V^j}(B_{20\rh}(z))\\
	&\le C_y(20\rh)^{-d_y}C_M\no{V^j}(M)(20\rh)^{n}.\\
	&\le 20^{n-d_y}C_yC_M\no{V^j}(M)\rh^{n-d_y}.
	\end{align*}
	Note that $n-d_y\ge 1$ since we're dealing with cohomogeneity at least $2.$ Since the proof in Step 2 of Section 4.2 in \cite{DT} works as long as the exponent on $\rh$ is larger than 0 in the above estimate, we're done.
	\subsection{Regularity of Equivariant Plateau Problem in $G$-tubes}
	In this section, we will prove the following proposition, in which we adapt an idea by Fleming and Lawson in \cite{HL}.
	
	Let $\Om$ be a smooth $G$-invariant open Caccioppoli set, whose boundary intersects $B^G_{\rh}(x)$ transversely for some small $\rh$. Define
	$$\mathcal{P}(B^G_{\rh}(x),\Om)=\{\Om'\s M:\Om'\sm B^G_{\rh}(x)=\Om\sm B^G_{\rh}(x), \Om'\text{ is a Caccioppoli set}\},$$ to be the class of sets of finite perimeter that coincide with $\Om$ outside of $U.$  
	\begin{prop}\label{keyprop}
		Let $\Om$ be as above and $\rh$ be small enough so that $\pd B_\rh^G(x)$ is well-defined. Then there exists a $G$-invariant Caccioppoli set $\Xi\in \mathcal{P}(B^G_{\rh}(x),\Om)$ minimizing the perimeter in $M$. Moreover, any such $G$-invariant minimizer $\Xi$ is, in $B_\rh(G)(x)$, an open set whose boundary is smooth outside of a singular set of Hausdorff dimension at most $n-7.$
	\end{prop}
	\begin{proof}
		In the following reasoning, we will sometimes abuse the notation to use $U$ to denote the current $[[U]]$ associated with a Caccioppoli set $U$ to avoid cumbersome repetition of $[[]].$
		
		By Theorem 1.2 in \cite{DT}, it suffices to prove that there exists a $G$-invariant minimizer in $\mathcal{P}(B^G_{\rh}(x),\Om),$ since any minimizer will have the codimension $7$ regularity in $B_\rh^G(x)$.
		
		First, we can pick any minimizer $X$ in the class without the requirement of $G$-invariance. Such a minimizer exists by Theorem 1.2 in \cite{DT}. 
		
		Now, define 
		\begin{align*}
		f(x)=\int_G1_{X}(gx) d\mu(g).
		\end{align*}  
		By construction we have $0\le f\le 1.$ Since $f$ is defined by averaging lower-semicontinuous functions, we see that $f$ is also lower-semicontinuous by Fatou lemma. Thus, the sets
		$$X_\lam=f\m (\lambda,1],$$ are open. Moreover, $X_\lam$ is $G$-invariant because $f$ is. By construction, as long as $\lam<1,$ we have $X_\lam\sm B_\rh^G(x)=\Om\sm B_\rh^G(x).$  Thus $X_\lam$ belongs to the class $\mathcal{P}(B_\rh^G(x),\Om)$ for all $\lam<1.$
		
		Let $X_\lam=X_\lam\ip B^G_\rh(x)$, and $E_f=f(x)\lp d\textnormal{Vol}_M .$ $\ma(E_f)$ is clearly bounded by the volume of $M$. For any $n-1$ form $\om$ on $M,$ we have
		\begin{align}
		\pd E_f(\om)=&\int_M\int_G 1_X(gx) d\om d\mu(g)\vol_M(x)=\int_G\int_M  1_{g\m X}(x)d\om d\vol_Md\mu\\=&\int_G \pd[[1_{g\m X}\lp d\vol_M]](\om) d\mu=\int_G \pd g\pf X (\om) d\mu.
		\end{align}
		Thus, we have $\ma(\pd E_f)\le \int_G\ma(\pd [[1_{g\m X}\lp d\vol_M]])d\mu\le \ma(\pd X)$ by lowersemi continuity of mass and the fact that every $g$ acts as an isometry. This implies $E_f$ is a normal current, and thus 
		By 4.5.9(12) in \cite{HF}, $X_\lam$ are integral currents and Caccioppoli sets (interpreted in the right sense).
		
		In the following paragraphs, we will prove that $X_{\lam}$ is a $G$-invariant minimizer in the class $\mathcal{P}(B_\rh^G(x),\Om).$ 
		
		By 4.5.9 (13) in \cite{HF} and identity (4.2), we have
		\begin{align*}\int_G g\pf Xd\mu(g)=	E_f=\int_0^1 X_\lam d\lam,
		\end{align*}and thus
		\begin{align*}
		\int_G\pd g\pf Xd\mu(g)=\pd E_f=\int_0^1 \pd X_\lam d\lam
		\end{align*}
		With gradient variational formula 4.5.9(13) in \cite{HF}, this immediately yields
		\begin{align*}
		\int_{[0,1]}\ma(\pd X_\lam)d\lam=\ma(\pd E_f)\le \ma(\pd X).
		\end{align*}
		Since $\pd X_\lam$ are integral currents, we have $\ma(\pd X_\lam)\le\ma(\pd X)$ by definition of a minimizer. This implies that $\ma(\pd X_\lam)=\ma(\pd X)$ for almost every $\lam$, and $X_\lam$ are perimeter minimizers for almost every $\lam$. Thus, $X_{\lam}$ is a $G$-invariant minimizer in the class $\mathcal{P}(B_\rh^G(x),\Om).$	 
	\end{proof}
	\begin{cor}\label{gta}
		Let $\Om$ be a smooth $G$-invariant open Caccioppoli set, whose boundary intersects $B^G_{\rh}(x)$ transversely. If a $G$-invariant Caccioppoli set $\Xi^G$ minimizes perimeter among $G$-invariant elements in $ \mathcal{P}(B^G_{\rh}(x),\Om).$ Then $X$ minimizes the perimeter in $ \mathcal{P}(B^G_{\rh}(x),\Om),$ without the restriction to $G$-invariant elements. Moreover, $\Xi$ is, in $U$, an open set whose boundary is smooth outside of a singular set of Hausdorff dimension at most $n-7.$
	\end{cor}
	\begin{proof}
		Pick a $G$-invariant minimizer $\Xi$ in Proposition \ref{keyprop}. Since $\Xi^G$ is minimizing among all $G$-invariant competitors, we have $$\text{Per}(\Xi^G,B_\rh^G(x))\le \text{Per}(\Xi,B_\rh^G(x)).$$ However, $\Xi$ is minimizing among all competitors, not only the invariant ones, so this immediately implies $\Xi^G$ is a minimizer among all competitors. The regularity result follows from \cite{HF1}
	\end{proof}
	\subsection{Proof of Lemma \ref{l4.2}}
	The same proof as Section 4.3 of \cite{DT} applies up until the last sentence in that section. We can use the same argument to show that $\pd\ti{\Om}^j$ is an area minimizer in $B^G_{\rh/2}(y)$ with respect to $G$-invariant competitors. For approximating $\Xi^{j,k}$ by invariant smooth functions, we can do this by the following averaging construction. Let $$G(f)(x)=\int_G f(g.x) d\mu(g).$$
	$G(f)$ is smooth by construction if $f$ is smooth. Now, just take a smooth approximation $f_n$ of $1_\Xi^{j,k}$ in the sense that $f_n\to 1_{\Xi^{j,k}}$ in $L^1$ and $\textnormal{Var}(f_n,M)\to \textnormal{Per}(\Xi^{j,k}).$ Then up to subsequence, $G(f_n)$ is also a smooth approxmiation of $1_{\Xi^{j,k}}.$ Then level sets of $G(f_n)$ would provide $G$-invariant smooth approximations of $\Xi^{j,k}.$ The rest of the argument goes through up until the last sentence about stability and stationarity of $\pd\ti{\Om}^j.$
	
	All our construction and competitors are $G$-invariant. To get rid of $G$-invariance restriction on minimizing, we invoke Corollary \ref{gta} to show that the minimizer is minimizing among all competitors. Then to prove stability, first note that $\pd\ti{\Om^j}$ is stable with respect to $G$ invariant vector fields by the same argument as in the non-invariant case in \cite{DT}. To show this, note that stationarity of $\pd\ti{\Om^j}$ comes from the locally minimizing property in small tubes as we have just shown and an invariant partition of unity. If $\pd\ti{\Om^j}$ is not stable with respect to $G$-invariant vector fields, then take a $G$-invariant vector field that gives a negative second variation on the regular part. Flow $\pd\ti{\Om^j}$ according to this vector field for a small time to get a family of open sets which lies in $\mathcal{H}(\Om^j,\an)$ yet ends up with smaller area than $\pd\ti{\Om^j}$ by the second variation. This is impossible as it contradicts the minimality of $\pd\ti{\Om^j}$ in $\mathcal{H}(\Om^j,\an)$. Now we claim that the first eigenfunction on the regular part must be equivariant, since composition with any $g$ gives the same Rayleigh quotient and the first eigenvalue is simple. Thus, considering the Rayleigh quotient of invariant functions and the Rayleigh quotient among all functions, we see that the minimum of the two coincide and conclude it's stable on the regular part, which suffices for the compactness and regularity theorem in \cite{SS}. This is also the argument used by Proposition 4.6 in \cite{DK}.
	\section{Regularity of $G$-Varifolds with replacements in $G$-annuli}\label{conc}
	This section corresponds to Section 5 of \cite{DT}. 
	
	We apply Proposition \ref{p2.6} three times as in Section 2.4 of \cite{DT} to obtain
	\begin{prop}\label{p2.7}(Compare Proposition 2.7 in \cite{DT}.)
		Let $V$ and $r$ be as in Proposition \ref{p2.6}. Fix $x\in M$ and $\an\in\ann_{r(x)}(x)$. Then\begin{itemize}
			\item[(a)] $V$ has a replacement $V'$ in $\an$ such that,		\item[(b)] $V'$ has a replacement $V''$ in any
			\begin{align*}
			\an_{'}\in \an_{r(x)}(x)\cup \bigcup_{y\not\in G.x}\ann_{r'(y)}(y),
			\end{align*}
			\item[(c)] $V''$ has a replacement $V'''$ in any $\an_{''}\in \an_{r''(y)}(y)$ with $y\in M,$ where $r',r''$ are both positive functions.
		\end{itemize}
	\end{prop}
	This section will be dedicated to prove the following proposition
	\begin{prop}\label{p2.8}(Compare Proposition 2.8 in \cite{DT}.)
		Let $V$ be as in Proposition \ref{p2.7}. Then $V$ is induced by a generalized minimal hypersurface $\Si$ in the sense of Definition \ref{d1.5}. 
	\end{prop}
	\subsection{Tangent cones}
	\begin{lem}\label{l5.2}
		(Compare Lemma 5.2 in \cite{DT}.) Let $V$ be a stationary $G$-varifold in an open $G$-set $U\s M$ having a $G$-invariant replacement in any annulus $\an\in\ann_{r(x)}(x)$ for some positive function $r.$ Then
		\begin{itemize}[nosep]
			\item $V$ is integer rectifiable;
			\item $\ta(x,V)\ge 1$ for any $x\in U$;
			\item any tangent Cone $C$ to $V$ at $x$ is a minimal generalized hypersurface for general $n$ and (a multiple of) a hyperplane for $n\le 6$ or $\dim G.x\ge n-6.$
		\end{itemize}
	\end{lem}
	\begin{proof}
		First, fix $x\in\text{supp}\no{V}$ and $0<r<\min\{\frac{r(x)}{20},\text{Inj}(M)/4\}.$ Let $V'$ be the replacement of $V$ in the annulus $\an(x,r,2r).$ We claim that $\no{V'}\not\equiv 0$ on $\an(x,r,2r).$ If that's note the case, then there would be $\rh\le r$ and $\e$ such that $\textnormal{supp}(\no{V'})\cap \pd B^G_\rh(x)\not=\es$ and $\textnormal{supp}(\no{V'})\cap \an(x,\rh,\rh+\e)=\es.$ By choice of $\rh$ this would contradict Theorem 5.1(ii) in \cite{DT}. For the assumption of convexity in that theorem, we only need to use Theorem 5.1(ii) in \cite{DT} locally, so we can choose $r$ small enough to make the tubes very close to cylinders in local charts, which would satisfy the convexity assumption required in the proof of that Theorem. Alternatively, as remarked in the footnote 2 on page 425 of \cite{BW}, a local version of the theorems in \cite{BW} also holds, so we can use the maximum principle Theorem 1 or 5 in \cite{BW} to prove Theorem 5.1(ii) for our case.
		
		Thus, $V'\ip \an(x,r,2r)$ is a non-empty minimal generalized hypersurface. Thus, there exists $y\in \an(x,r,2r)$ with $\ta(y,V')\ge 1.$ Without loss of generality, we can assume $y\in B_{2r}(x)\sm B_r(x)$. (By $G$-invariance, we can always use some $g.y$ to substitute $y$ that sits in this geodesic annulus). By applying Lemma \ref{la1} and letting $20r<\rh_0,$ and $G$-invariance of $V$, we have
		\begin{equation}\label{e5.1}
		\begin{split}
		\no{V}(B_{4r}^G(x))&\le \no{V}\left(\bigcup_{z\in \mathcal{B}}B_{20r}(z)\right)\\
		&\le |\mathcal{B}|\no{V}(B_{20r}(x)).
		\end{split}
		\end{equation}By definition of replacements, we have
		\begin{align}\label{e5.2}
		\no{V}(B^G_{4r}(x))=\no{V'}(B^G_{4r}(x)).
		\end{align}
		Note that the collection $\mathcal{B}$ of balls is disjoint in Lemma \ref{la1}, so by $G$-invariance of $V'$ we have
		\begin{equation}\label{e5.3}
		\begin{split}
		\no{V'}(B_{4r}^G(x))\ge \no{V'}\left(\bigcup_{z\in \mathcal{B}}B_{4r}(z)\right)\ge |\cb|\no{V'}(B_{4r}(x))\ge|\cb|\no{V'}(B_{2r}(y)).
		\end{split}
		\end{equation}
		Combining (\ref{e5.1}), (\ref{e5.2}), and (\ref{e5.3}), and dividing by $|\cb|(20r)^n$ we deduce that
		\begin{align*}
		\frac{\no{V}(B_{20r}(x))}{(20r)^n}\ge 10^n \frac{\no{V'}(B_{2r}(y))}{(2r)^n}.
		\end{align*}
		By the monotonicity formula and $\ta(y,V')\ge 1$ we can deduce that there exists constant $C_M\ge 0$ so that $$\no{V'}(B_{2r}(y))\ge C_M\m\ta(y,V')(2r)^n.$$
		Thus, we have
		\begin{align*}
		\frac{\no{V}(B_{20r}(x))}{(20r)^n}\ge 10^n C_M\m\ta(y,V')\ge 10^n C_M.
		\end{align*}
		This implies $\ta(x,V)$ is uniformly bounded away from $0$ on $\text{supp}(\no{V})$ and Allard's rectifiability theorem (5.5 in \cite{WA}) implies that $V$ is rectifiable. 
		
		Use $C$ to denote the tangent cone to $V$ at $x$ and let $\rh_k\to 0$ a sequence with $V_{\rh_k}^x\to C.$ By $G$-invariance, the pushforward by any $g\in G$ of these cones $dg( C)$ and blowing up sequence  $dg(V_{\rh_k}^x)$ are also tangent cones and blowing ups at $g.x$. The rest of the proof goes the same as proof of Lemma 5.2 in \cite{DT}. One difference is that we have to substitute annulus in $T_x M$ with annulus around $i\pf T_y G.y,$ the tangent to the orbit $G.y,$ i.e., 
		\begin{align*}
		\textnormal{ann}(r,s,T_yG.y)=\{v\in T_yM|r<\text{dist}_{T_yM}(v,T_yG.y)<s\}.
		\end{align*} 
		The other is deducing the regularity of the cone $C$ in $T_yG.x.$ By Lemma \ref{stc}, $C=T_yG.x\times W$ where $W$ is supported in $(T_yG.x)^\perp.$ By blowing up a corresponding sequence of replacements as in formula (5.3) to (5.5) \cite{DT}. We can deduce that $C=T_yG.x\times W$ and thus $W$ is a stable cone. Since $W$ lies in $(T_yG.x)^\perp,$ which has dimension $n-\dim G.x,$ we deduce that $W$ is a multiple of plane if $3\le$$n-\dim G.x\le 6,$ and have singularity of Hausdorff dimension at most $n-\dim G.x-7$ if $n-\dim G.x\ge 7.$  { Only in the final part of the argument do we need to impose cohomogeneity larger than $2$, where we have used Simons' theorem on stable cones to $W$.}
	\end{proof}
	\subsection{Proof of Proposition \ref{p2.8}}
	\se{Step 1 to 4}
	The same with Section 5.4 in \cite{DT}. 
	\se{Step 5}
	The center of our $B_r^G(x)$ are the orbits $G.x$ instead of points. If no orbits are of dimension at least $n-6,$ then we're fine. However, for those of dimension at least $n-6,$ then we shall invoke the Lemma 5.2 to deduce that the tangent cones are still hyperplanes. The same reasoning in Section 5.4 of \cite{DT} dealing with dimension lower than $7$ can be used in our case to deduce the proposition. {We can also invoke the recently proved sharp unique continuation and sharp maximum principles of \cite{NW2} to to substitute the PDE arguments used in Section 5.4 of \cite{DT}, so that we do not need to set up normal coordinates. The main conclusion of Step 2 can be deduced directly from the unique continuation Theorem 1.2 in \cite{NW2}. In Step 3, notice that both $\Ga'$ and $\Ga''$ are smooth up to codimension $7$ sets and possibly $\ga.$ By transversality assumption, $\Ga'$ is smooth along $\ga.$ However in Step 2, we proved that the tangent cones are $\ga$ are multiplicity $1$ planes so $\Ga''$ is smooth along $\ga$ by Allard regularity. Thus, again we can use Theorem 1.2 in \cite{NW2} to glue $\Ga'$ and $\Ga''$. In Step 5 we can use the maximum principle Theorem 1.1 in \cite{NW2} to substitute the one used in \cite{DT}.}
	\section{Cohomogeneity $1$ case}\label{coho1}
	By the classification in \cite{PM}, we have either $M/G=S^1$ or $M/G=[-1,1].$  Let $\Si_t=\pi\m(t)$ where $t\in [-1,1]$ and we identify $S^1$ with $[-1,1]/\-1\sim 1\}.$ By Proposition 1 in \cite{TP}, $\hm^n(\Si_t)$ depends smoothly on $t$ on principal orbits and extends continuously to $0$ on exceptional orbits. If $M/G=S^1,$ then every orbit is principal, so we must have a critical point, which gives a smooth minimal generalized hypersurface. If $M/G=[-1,1]$, then $\hm^n(\Si_{-1})=\hm^n(\Si_1)=0,$ and $\hm^n(\Si_t)$ is smooth on $(-1,1)$ and continuous on $[-1,1]$, so there must exists a critical point. 
	
	\section*{Acknowledgements}
	The author is very fortunate to have been introduced to the world of geometric measure theory by Professor William Allard, who has read the very first draft of this paper and given invaluable suggestions. The author cannot thank him enough for many stimulating and revelatory conversations about GMT. Also, the author would like to thank Professor Robert Bryant for countless helpful discussions on Lie groups, transformation group theory and Riemannian geometry, and his unwavering support. Indeed, this paper originates from a conversation about invariant cycles in Lie groups with Professor Bryant. The author also thanks Professor Hubert Bray for many helpful meetings and constant encouragement. He would like to thank Professor
	Camillo De Lellis and Fernando Cod\'a Marques for their interest in this work and pointing out the reference \cite{DK}. {The author also would like to thank the referee for countless helpful advice on both the structure of the paper and the writing of some proofs, which improves the readability of the paper by a large margin. A special thanks goes to Antoine Song for pointing out that the regularity argument does not work for cohomogeneity 2 case.} Last but not least, he is indebted to Professor David Kraines, who partially funded the research in the paper with PRUV Fellowship.
	
	\appendix
	\addcontentsline{toc}{section}{Appendices}
	\renewcommand{\thesection}{\Alph{section}}
	\section{Appendix}
	\subsection{Ball covering of tubes}
	We will prove the following useful lemma.
	\begin{lem}\label{la1}
		For any $y\in M,$ there exists $\rh_0>0$ so that for any $\rh<\rh_0,$ there exists a collection $\mathcal{B}$ of disjoint geodesic balls of radius $\rh$ with centers in $G.y$ so that the concentric balls with radius $5\rh$ covers $B^G_\rh(y)$. Moreover, the number of balls in this collection is at most $C_y \rh^{-d_y},$ where $C_y$ is a constant depending only on $G.y,$ and $d_y=\dim G.y.$
	\end{lem}
	\begin{proof}
		Here we use a basic 5-times-radius covering theorem (putting $\tau=2$ and $\de$ as diameter in 2.8.5 in \cite{HF}), that says for a covering using metric balls in metric space, we can find a disjoint subcollection so that 5-times-radius concentric balls of this subcollection would cover all the original balls. Now consider the covering of $B_{\rh}^G(y)$ by $\{B_{\rh}(z)|z\in G.y\}.$ We deduce that there exists a set $\mathcal{B}$ consisting of finitely many points so that $B_{\rh}(z)\cap B_{\rh}(z')=\es$ if $z\not=z',z,z'\in \mathcal{B}$ and 
		\begin{align*}
		B^G_{\rh}(y)\s \bigcup_{z\in \mathcal{B}}B_{5\rh}(z).
		\end{align*}
		Note that the cardinality of $\mathcal{B}$ satisfies the following obvious bound
		\begin{align*}
		|\mathcal{B}|\le \frac{\vol(B_{\rh}^G(y))}{B_{\rh}(y)},
		\end{align*}since $G$ acts by isometries and thus pushes forward geodesic balls to geodesic balls.
		
		Let $$d_y=\dim G.y.$$ Recall the volume of tubes in \cite{AG}. There exists $\rh_0>0$ so that for all $\rh<\rh_0$, we would have $$\vol(B^G_{\rh}(y))\le C_{d_y} \hm^{d_y}(G.y)\rh^{n+1-d_y},$$ for some dimensional constant $C_{d_y}>0.$ Moreover, by the volume of geodesic balls, we could assume that
		$
		\vol(B_{\rh}(y))\ge C_{n}\rh^{n+1},
		$
		for some dimensional constant $C_n>0$ by shrinking $\rh_0$ if necessary. If $\rh<\rh_0,$ then there exists $C_y>0$ depending only on $G.y$ and $M$ such that 
		\begin{align*}
		|\mathcal{B}|\le C_y \rh^{-d_y}.
		\end{align*}
		
	\end{proof}
	\subsection{Splitting of Tangent Cone of Integral $G$-varifold}\label{sotc}
	Let $G.x$ be an orbit of dimension $d_x$, and $B_\rh^G(x)$ be the $\rh$-tubular neighborhood around $x$. Suppose $V$ is a rectifiable $G$-varifold in $\mathbf{V}_n$ and $x$ is in $\text{spt}V$. We will prove the following lemma which implies that the tangent cone splits as a product into normal directions and tangential directions to $G.x.$
	\begin{lem}\label{stc}
		For any point $y\in G.x,$ there exists a tangent cone $C_y\s T_y M$ of $V$, so that $C_y+w=C_y$ for any $w\in T_y G.x\s T_y M.$
	\end{lem}
	\begin{proof}Without loss of generality, we can assume $y=x,$ since we can always pushforward our constructions by any element of $g.$ We will use $\exp$ to denote the restriction of exponential map in $T_yM$ inside a ball of injectivity radius. Let $r_i\to\infty$ be a sequence so that $(r_i )\pf \exp\m\pf (V)\to C$ as varifold, where $r_i$ is multiplication by $r_i$ in $T_yM.$ Note that we have $$(r_i)\pf \exp\m\pf G.y=i\pf T_y G.y,$$ if $i$ is the inclusion $G.y\hookrightarrow M.$
		
		By Main Theorem of \cite{MS}, we can isometrically embed $M$ into some $\R^N$ so that the action of $G$ on $M$ comes from a linear representation of $G$ on $\R^N.$ We will also denote this action as $\rh(g)z$ for $z\in\R^N.$ We will identify $M$ as a submanifold of $\R^N$ in the following reasoning.
		
		Let $c\in C_y$ be a point in the tangent cone. We will also regard it as a vector. We can find a sequence of points $c_j\in V$ so that $r_j\exp\pf\m c_j\to c.$ Let $g(t)$ be a smooth path in $G$ so that $g(0)=0$, $$\frac{d}{dt}\bigg|_{t=0}g(t).y=w.$$ Such a path exists by lifting a corresponding path starting with $w\in T_y G.y$ and staying in $G.y\approx G/G_y$. 
		
		Now, note that $g(r_i\m).c_i\in V.$ If we can prove that
		\begin{equation}\label{sfm}
		r_i\exp\m\pf (g(r_i\m).c_i)=c+w,
		\end{equation}
		then we are done.
		To prove this, we need to compare $\exp\pf\m(z)$ with $z-y$. First, note that $d(z-y)|_{T_y\R^N}=\id_{T_y\R^N}=d\exp\pf \m(z)|_{T_y\R^N}.$
		Thus, $\exp\pf\m(z)-(z-y)=O(d_M(z,y)^2)$ for $z\in M.$. Thus, we have
		\begin{align*}
		&r_i\exp\m\pf (g(r_i\m).c_i)\\=&r_i(g(r_i\m).c_i-y)+r_iO(d_M(g(r_i\m).c_i,y)^2)\\
		=&r_i(\rh(g(r_i\m))c_i-\rh(g(r_i\m))y+\rh(g(r_i\m))y-y)+r_iO(\no{\exp\m\pf(g(r_i\m).c_i)}^2)\\
		=&\rh(g(r_i\m))r_i\exp\m(c_i)+r_iO(\no{\exp\m (g(r_i\m).c_i)}^2)\\
		&+r_i(\rh(g(r_i\m))-\rh(g(0)))y+r_iO(\no{\exp\m g(r_i\m)y})+r_iO(\no{\exp\m\pf(g(r_i\m).c_i)}^2)\\
		=&\rh(g(r_i\m))r_i\exp\m(c_i)+r_i(\rh(g(r_i\m))-\rh(g(0)))y+O(r_i\m).
		\end{align*}
		Let $i\to\infty$, and we immediately get \ref{stc}.
	\end{proof}
	\begin{cor}\label{tstc}
		The support of any such cone $C_y$ as in \ref{stc} is a product of $T_yG.x$ and a rectifiable set $W$ supported in $i\pf(T_yG.x)^\perp.$
	\end{cor}
	\begin{proof}
		Take $W=C_y\cap i\pf(T_yG.x)^\perp$ and use Lemma \ref{stc}.
	\end{proof}

\end{document}